\documentclass[reqno,english]{amsart}
\usepackage[T1]{fontenc}
\usepackage[latin9]{inputenc}
\usepackage{color}
\usepackage{babel}
\usepackage{refstyle}
\usepackage{mathrsfs}
\usepackage{mathtools}
\usepackage{enumitem}
\usepackage{amstext}
\usepackage{amsthm}
\usepackage{amssymb}
\usepackage[all]{xy}
\usepackage[unicode=true,pdfusetitle,
 bookmarks=true,bookmarksnumbered=false,bookmarksopen=false,
 breaklinks=false,pdfborder={0 0 0},pdfborderstyle={},backref=false,colorlinks=false]
 {hyperref}
\hypersetup{
 colorlinks=true,citecolor=blue,linkcolor=blue,linktocpage=true}

\makeatletter

\AtBeginDocument{\providecommand\thmref[1]{\ref{thm:#1}}}
\AtBeginDocument{\providecommand\lemref[1]{\ref{lem:#1}}}
\AtBeginDocument{\providecommand\corref[1]{\ref{cor:#1}}}
\RS@ifundefined{subsecref}
  {\newref{subsec}{name = \RSsectxt}}
  {}
\RS@ifundefined{thmref}
  {\def\RSthmtxt{theorem~}\newref{thm}{name = \RSthmtxt}}
  {}
\RS@ifundefined{lemref}
  {\def\RSlemtxt{lemma~}\newref{lem}{name = \RSlemtxt}}
  {}

\numberwithin{equation}{section}
\numberwithin{figure}{section}
\numberwithin{table}{section}
\theoremstyle{plain}
\newtheorem{thm}{\protect\theoremname}[section]
\theoremstyle{remark}
\newtheorem{rem}[thm]{\protect\remarkname}
\theoremstyle{plain}
\newtheorem{lem}[thm]{\protect\lemmaname}
\theoremstyle{plain}
\newtheorem{cor}[thm]{\protect\corollaryname}
\theoremstyle{plain}
\newtheorem{prop}[thm]{\protect\propositionname}
\theoremstyle{definition}
\newtheorem{defn}[thm]{\protect\definitionname}
\theoremstyle{definition}
\newtheorem{example}[thm]{\protect\examplename}
\theoremstyle{remark}
\newtheorem{conclusion}[thm]{\protect\conclusionname}
\theoremstyle{remark}
\newtheorem*{acknowledgement*}{\protect\acknowledgementname}

\providecommand{\MR}[1]{}

\usepackage{bbm}
\allowdisplaybreaks
\usepackage{needspace}

\newref{lem}{refcmd={Lemma \ref{#1}}}
\newref{thm}{refcmd={Theorem \ref{#1}}}
\newref{cor}{refcmd={Corollary \ref{#1}}}
\newref{sec}{refcmd={Section \ref{#1}}}
\newref{subsec}{refcmd={Section \ref{#1}}}
\newref{chap}{refcmd={Chapter \ref{#1}}}
\newref{prop}{refcmd={Proposition \ref{#1}}}
\newref{exa}{refcmd={Example \ref{#1}}}
\newref{tab}{refcmd={Table \ref{#1}}}
\newref{rem}{refcmd={Remark \ref{#1}}}
\newref{def}{refcmd={Definition \ref{#1}}}
\newref{fig}{refcmd={Figure \ref{#1}}}
\newref{claim}{refcmd={Claim \ref{#1}}}

\AtBeginDocument{
  
}

\makeatother

\providecommand{\acknowledgementname}{Acknowledgement}
\providecommand{\conclusionname}{Conclusion}
\providecommand{\corollaryname}{Corollary}
\providecommand{\definitionname}{Definition}
\providecommand{\examplename}{Example}
\providecommand{\lemmaname}{Lemma}
\providecommand{\propositionname}{Proposition}
\providecommand{\remarkname}{Remark}
\providecommand{\theoremname}{Theorem}

\begin{document}
\title{Infinite-dimensional stochastic transforms and reproducing kernel
Hilbert space}
\begin{abstract}
By way of concrete presentations, we construct two infinite-dimensional
transforms at the crossroads of Gaussian fields and reproducing kernel
Hilbert spaces (RKHS), thus leading to a new infinite-dimensional
Fourier transform in a general setting of Gaussian processes. Our
results serve to unify existing tools from infinite-dimensional analysis.
\end{abstract}

\author{Palle E.T. Jorgensen}
\address{(Palle E.T. Jorgensen) Department of Mathematics, The University of
Iowa, Iowa City, IA 52242-1419, U.S.A.}
\email{palle-jorgensen@uiowa.edu}
\author{Myung-Sin Song{*}}
\address{(Myung-Sin Song) Department of Mathematics and Statistics, Southern
Illinois University Edwardsville, Edwardsville, IL 62026, USA}
\email{msong@siue.edu}
\author{James Tian}
\address{(James F. Tian) Mathematical Reviews, 416 4th Street Ann Arbor, MI
48103-4816, U.S.A.}
\email{jft@ams.org}

\maketitle
\tableofcontents{}

\section{\label{sec:Introduction}Introduction}

Our present paper is motived by a host of applications of the general
theory of kernels, and their corresponding reproducing kernel Hilbert
spaces (RKHS). The list of these applications is long, and includes
both brand new advances, and classical topics in analysis. And the
following areas are included: Statistical inference, achieving separation
of features in machine learning designs, PDEs, harmonic analysis,
stochastic processes, mathematical physics and more\cite{MR4379621,MR4415393,MR4334244}.
While the various reproducing kernel Hilbert spaces involved constitute
versatile tools, their realization has a drawback. Its definition
and realization entails an abstract completion step. To overcome this,
in section \ref{sec:RKHS}, we present a general transform which applies
in the general context of reproducing kernel Hilbert spaces (RKHS)
but more concrete presentations. While in special cases, there are
transforms which produce \textquotedblleft concrete\textquotedblright{}
presentations for a particular RKHS, these transforms tend to be ad
hoc. For the benefit of the readers, we provide some example references:
The paper \cite{MR3857315} deals with sensitivity issues, \cite{MR3883202}
direct integration algorithms, \cite{4106847} image processing, \cite{MR2247587,MR3005666,MR3854652}
pattern recognition and machine learning, \cite{MR2177937} statistical
properties, and finally \cite{MR3534893} on selection of efficient
parameterizations. In addition to earlier papers by the co-authors
\cite{So07,JoSo07a}, we also include here a partial list of other
relevant and current citations; see e.g., the papers mentioned above,
\cite{MR3780557,MR3820672,MR3911884,MR1720704,MR3857315,MR3878657,MR3913046,MR3934645,MR3850675,MR3900802,MR3922239}.
Our present general transform will serve to unify special cases which
have appeared in the literature. In section \ref{sec:infinite}, we
present a new infinite-dimensional Fourier transform in a general
framework of Gaussian processes. The latter in turn is motivated by
applications to stochastic differential equations (SDE), including
generalized Ito calculus formulas. The links between the two main
results involving general transforms are related because of a general
fact about positive definite kernels (and the corresponding RKHSs):
Every positive definite kernel $K$, defined on $D\times D$ where
$D$ is a set, is the covariance kernel of a centered Gaussian process
indexed by $D$. The latter fact then unifies the two settings.

For background references on kernel theory and Gaussian processes,
we refer readers to the following \cite{MR2966130,MR562914,MR3843387,PaSc75},
but we caution that the literature is extensive. The book \cite{MR2239907}
offers a great overview of the subject and its applications. Our main
results presented below build on the following earlier papers \cite{jorgensen2019dimension,jst21pos,MR3843552,MR3888850,MR4020693}
by the co-authors. In addition, there are multiple current and very
active research teams which make use of diverse aspects of the kind
of kernel methods discussed here. We have been especially motivated
by the following recent papers, covering such diverse results as machine
learning, approximations, chaos, noisy data, Markov sampling, and
neural nets, all making use of analysis, algorithms, and optimization
with kernels and RKHSs, especially the papers \cite{MR2558684,MR2327597,MR2186447,MR4367288,MR4321422,MR4241215,MR4091197}.

\section{\label{sec:RKHS}A general transform for reproducing kernel Hilbert
spaces (RKHS).}

We present new frame-theoretic tools with view to learning and sampling
from large data-sets. A key role in this endeavor is played by our
use of infinite-dimensional Kaczmarz algorithms, and an analysis of
systems of projections. However a more realistic approach to an analysis
of large and non-linear data sets must be adapted to include an analysis
of \textquotedblleft noise\textquotedblright{} terms. More precisely,
in order to incorporate stochastic tools, one is faced with choices
among families of Gaussian processes, and probability spaces, and
Karhunen--Loève expansions which will then be adapted to the Hilbert
space analysis we introduced above. Please see \cite{jorgensen2019dimension,jst21pos,MR3843552,MR3888850,MR4020693,JoSo07a}
for examples. Such a stochastic path-space analysis is the focus of
the section below.

By a probability space, we mean a triple $\left(\Omega,\mathscr{F},\mathbb{P}\right)$
where 
\begin{itemize}
\item $\Omega$: set of sample points,
\item $\mathscr{F}$: $\sigma$-algebra of events (subsets of $\Omega$),
\item $\mathbb{P}$: a probability measure defined on $\mathscr{F}$. 
\end{itemize}
A random variable 
\begin{equation}
K:\Omega\rightarrow\mathbb{R}\:\left(\text{\ensuremath{\mathbb{C}}, or a Hilbert space}\right)
\end{equation}
is a measurable function defined on $\left(\Omega,\mathscr{F}\right)$,
i.e., we require that for Borel sets $B$ (in $\mathbb{R}$, or $\mathbb{C}$),
and cylinder sets (referring to a fixed Hilbert space $\mathscr{H}$)
we have $K^{-1}\left(B\right)\in\mathscr{F}$ where 
\begin{equation}
K^{-1}\left(B\right)=\left\{ \omega\in\Omega:K\left(\omega\right)\in B\right\} .
\end{equation}
The distribution $\mu_{K}$ of $K$ is the measure 
\begin{equation}
\mu_{K}\coloneqq\mathbb{P}\circ K^{-1}.
\end{equation}
If $\mu_{K}$ is Gaussian, we say that $K$ is a Gaussian random variable.
A Gaussian process is a system $\left\{ K_{x}:x\in X\right\} $ of
random variables (refer to $\left(\Omega,\mathscr{F},\mathbb{P}\right)$),
indexed by some set $X$. 

Here we shall restrict to the case of a Gaussian process, and we shall
assume 
\begin{equation}
\mathbb{E}\left(K_{x}\right)=0,\quad\forall x\in X;
\end{equation}
where 
\begin{equation}
\mathbb{E}\left(\cdot\right)=\int_{\Omega}\left(\cdot\right)d\mathbb{P}
\end{equation}
denotes expectation w.r.t. $\mathbb{P}$.

If $\mu_{K}\in N\left(0,1\right)$, i.e., 
\begin{equation}
\mu_{K}\left(t\right)=\frac{1}{\sqrt{2\pi}}e^{-t^{2}/2},\quad t\in\mathbb{R},
\end{equation}
we say that $K$ (or $\mu_{K}$) is a standard Gaussian. 

\subsection{Effective sequences}

We now turn to a general framework dealing with an approximation-algorithm
but couched in the context of projections in Hilbert space. Starting
with a sequence $(P_{n})$ of projections, in the premise of Theorem
\ref{thm:pc3}, we define an associated approximation-algorithm. Our
algorithm is in terms of a new system, defined as a sequence of operator
products (\ref{eq:pc4}). And we then present necessary and sufficient
conditions for the algorithm to converge. When it does, we say that
the system is effective, see (\ref{eq:pc3}) and (\ref{eq:pc5}).
This general setting and framework are inspired by (but different
from) such well known algorithms as Gram-Schmidt, and Kaczmarz, see
also \cite{MR3996038,MR3796641}.

Our first applications of effective sequences, Corollary \ref{cor:Qn},
Proposition \ref{prop:Qn}, are to stochastic processes. In Theorem
\ref{thm:kol} (Kolmogorov) we recall the correspondence between positive
definite kernels on the one hand, and Gaussian processes, Gaussian
fields on the other. Our discussion of transforms will then resume
in sections \ref{subsec:Stochastic-analysis-and} and \ref{subsec:The-isomorphism}
below.
\begin{thm}
\label{thm:pc3}Let $\left\{ P_{j}\right\} _{j\in\mathbb{N}_{0}}$
be a sequence of orthogonal projections in a Hilbert space $\mathscr{H}$.
Set 
\begin{align}
T_{n} & =\left(1-P_{n}\right)\left(1-P_{n-1}\right)\cdots\left(1-P_{0}\right),\label{eq:pc3}\\
Q_{n} & =P_{n}\left(1-P_{n-1}\right)\cdots\left(1-P_{0}\right),\label{eq:pc4}
\end{align}
where $Q_{0}=P_{0}$. 

For all $n\in\mathbb{N}$, we have 
\[
\left\Vert x\right\Vert ^{2}=\left\Vert T_{n}x\right\Vert ^{2}+\sum_{k=0}^{n}\left\Vert Q_{k}x\right\Vert ^{2},\quad x\in\mathscr{H}.
\]
 The combination of \ref{eq:pc3} and \ref{eq:pc5} yields a condition
for when the algorithm yields a direct sum representation for the
Hilbert space $\mathscr{H}$. Equation \ref{eq:pc5} expresses this
in the form of a resolution of the identity operator for $\mathscr{H}$
in terms of the generalized Kaczmarz operators $Q_{n}$ introduced
in \ref{eq:pc4}. Equations \ref{eq:pc6} and \ref{eq:pc7} yield
equivalent forms of this sum expansion. Hence $T_{n}\xrightarrow{\;s\:}0$
if and only if 
\begin{equation}
I=\sum_{j\in\mathbb{N}_{0}}Q_{j}^{*}Q_{j}.\label{eq:pc5}
\end{equation}
More precisely, (\ref{eq:pc5}) means that, 
\begin{equation}
\left\langle x,y\right\rangle =\sum_{j\in\mathbb{N}_{0}}\left\langle Q_{j}x,Q_{j}y\right\rangle ,\quad x,y\in\mathscr{H}.\label{eq:pc6}
\end{equation}
In particular, 
\begin{equation}
\left\Vert x\right\Vert ^{2}=\sum_{j\in\mathbb{N}_{0}}\left\Vert Q_{j}x\right\Vert ^{2},\quad x\in\mathscr{H}.\label{eq:pc7}
\end{equation}
\end{thm}

\begin{proof}
Note that 
\begin{eqnarray*}
\left\Vert T_{n}x\right\Vert ^{2} & = & \left\Vert \left(1-P_{n}\right)\left(1-P_{n-1}\right)\cdots\left(1-P_{0}\right)x\right\Vert ^{2}\\
 & = & \left\Vert \left(1-P_{n-1}\right)\cdots\left(1-P_{0}\right)x\right\Vert ^{2}-\left\Vert P_{n}\left(1-P_{n-1}\right)\cdots\left(1-P_{0}\right)x\right\Vert ^{2}\\
 & = & \left\Vert T_{n-1}x\right\Vert ^{2}-\left\Vert Q_{n}x\right\Vert ^{2}\\
 & = & \left\Vert T_{n-2}x\right\Vert ^{2}-\left\Vert Q_{n-1}x\right\Vert ^{2}-\left\Vert Q_{n}x\right\Vert ^{2}\\
 & \vdots\\
 & = & \left\Vert \left(1-P_{0}\right)x\right\Vert ^{2}-\left\Vert Q_{1}x\right\Vert ^{2}-\cdots-\left\Vert Q_{n-1}x\right\Vert ^{2}-\left\Vert Q_{n}x\right\Vert ^{2}\\
 & = & \left\Vert x\right\Vert ^{2}-\left\Vert Q_{0}x\right\Vert ^{2}-\left\Vert Q_{1}x\right\Vert ^{2}-\cdots-\left\Vert Q_{n-1}x\right\Vert ^{2}-\left\Vert Q_{n}x\right\Vert ^{2}.
\end{eqnarray*}
Therefore 
\[
T_{n}\xrightarrow{\;s\:}0\Longleftrightarrow\left\Vert x\right\Vert ^{2}=\sum_{j\in\mathbb{N}_{0}}\left\Vert Q_{j}x\right\Vert ^{2}.
\]
\end{proof}
\begin{rem}
The system of operators $\left\{ Q_{j}\right\} _{j\in\mathbb{N}_{0}}$
in \thmref{pc3} has frame-like properties, see (\ref{eq:pc5})--(\ref{eq:pc7}).
Specifically, the mapping 
\[
\mathscr{H}\ni x\xmapsto{\;V\;}\left(Q_{j}x\right)\in l^{2}\left(\mathbb{N}_{0}\right)\otimes\mathscr{H}
\]
plays the role of an analysis operator, and the synthesis operator
$V^{*}$ is given by 
\[
l^{2}\left(\mathbb{N}_{0}\right)\otimes\mathscr{H}\ni\xi\xmapsto{\;V^{*}\;}\sum_{j\in\mathbb{N}_{0}}Q_{j}^{*}\xi_{j}.
\]
Note that $1=V^{*}V$, and (\ref{eq:pc7}) is the \emph{generalized
Parseval identity}. 
\end{rem}

\begin{lem}
\label{lem:Let--be}Let $\left\{ Z_{n}\right\} _{n\in\mathbb{N}_{0}}$
be a system of independent identically distributed $N\left(0,1\right)$s
(i.i.d $N\left(0,1\right)$) on $\mathbb{R}.$ Let $\mathscr{H}$
be a Hilbert space, and $\left\{ Q_{n}\right\} _{n\in\mathbb{N}_{0}}$
a system of projections as in \thmref{pc3}, i.e., 
\begin{equation}
\sum_{n\in\mathbb{N}_{0}}\left\langle Q_{n}u,Q_{n}v\right\rangle _{\mathscr{H}}=\left\langle u,v\right\rangle _{\mathscr{H}},\quad\forall u,v\in\mathscr{H}.\label{eq:G7}
\end{equation}
Then 
\begin{equation}
W\left(\cdot\right)=W^{\left(Q,\mathscr{H}\right)}\left(\cdot\right)\coloneqq\sum_{n\in\mathbb{N}_{0}}Q_{n}Z_{n}\left(\cdot\right)\label{eq:G8}
\end{equation}
 defines an operator valued Gaussian process, and 
\begin{equation}
\mathbb{E}\left(\left\langle W\left(\cdot\right)u,W\left(\cdot\right)v\right\rangle _{\mathscr{H}}\right)=\left\langle u,v\right\rangle ,\quad\forall u,v\in\mathscr{H}.\label{eq:G9}
\end{equation}
\end{lem}

\begin{proof}[Proof sketch]
Fix $u,v\in\mathscr{H}$; then 
\begin{eqnarray*}
\text{LHS}_{\left(\ref{eq:G9}\right)} & \underset{\text{by \ensuremath{\left(\ref{eq:G8}\right)}}}{=} & \underset{\mathbb{N}_{0}\times\mathbb{N}_{0}}{\sum\sum}\left\langle Q_{n}u,Q_{m}v\right\rangle _{\mathscr{H}}\underset{=\delta_{n,m}}{\underbrace{\mathbb{E}\left(Z_{n}Z_{m}\right)}}\\
 & = & \sum_{n\in\mathbb{N}_{0}}\left\langle Q_{n}u,Q_{n}v\right\rangle _{\mathscr{H}}\\
 & \underset{\text{by \ensuremath{\left(\ref{eq:G7}\right)}}}{=} & \left\langle u,v\right\rangle _{\mathscr{H}}.
\end{eqnarray*}
\end{proof}
\begin{cor}
\label{cor:Qn}Let $\left\{ Q_{n}\right\} _{n\in\mathbb{N}_{0}}$
be an effective system in $\mathscr{H}_{K}$, where $K:X\times X\rightarrow\mathbb{R}$
is a given p.d. kernel and $\mathscr{H}_{K}$ the associated RKHS.
Then $W$ from (\ref{eq:G8}) has the property that 
\[
K\left(x,y\right)=\mathbb{E}\left(\left\langle W\left(\cdot\right)K_{x},W\left(\cdot\right)K_{y}\right\rangle _{\mathscr{H}_{K}}\right).
\]
\end{cor}

We recall the following theorem of Kolmogorov. It states that there
is a 1-1 correspondence between p.d. kernels on a set and mean zero
Gaussian processes indexed by the set. One direction is easy, and
the other is the deep part:
\begin{thm}[Kolmogorov]
\label{thm:kol} Let $X$ be a set. A function $K:X\times X\rightarrow\mathbb{C}$
is positive definite if and only if there is a Gaussian process $\left\{ W_{x}\right\} _{x\in X}$
realized in $L^{2}\left(\Omega,\mathscr{F},\mathbb{P}\right)$ with
mean zero, such that 
\begin{equation}
K\left(x,y\right)=\mathbb{E}\left[\overline{W}_{x}W_{y}\right].\label{eq:K1}
\end{equation}
Note: This deals with the general framework of Gaussian processes:
There are two sets where $X$ is the indexed set for the Gaussian
processes, for the classical case, $X$ may be a real line. And $\Omega$
is the probability space, also referred to as the sample space.
\end{thm}

Starting with a set $X$, and a positive definite (p.d) kernel on
$X\times X$, there are then two approaches to fleshing out the desired
Gaussian process $W=W^{(K)}$ indexed by $X$, centered, and having
$K$ as its covariance kernel. Both are useful, and they serve different
purposes.

The first approach is based on an application of the Kolmogorov extension
principle; and its validity can be established with the use of the
assumed p.d. property for $K$ ; we refer to \cite{MR4379621,MR2966130,MR1474726,MR3843552}.

The other approach is different, stating with a fixed p.d. kernel
$K$ as above, one then proceeds as follows: (i) pass to the corresponding
reproducing kernel Hilbert space (RKHS) $H\left(K\right)$ ; (ii)
pick an orthonormal basis (ONB) in $H\left(K\right)$, and (iii) build
the desired Gaussian process $W=W^{(K)}$ from it with an application
of a choice of an i.i.d. system of $N(0.1)$ random variables; and
(iv) an application of the Central Limit Theorem; see e.g., \cite{MR2239907,MR562914,MR4020693,MR2558684}.
In both cases, we make use of the general fact that a centered Gaussian
process is determined by its covariance kernel, see e.g., \cite{MR3888850}.
For the details of this approach, we also refer to Lemma \ref{lem:Let--be}
above, especially equations (\ref{eq:G8}) and (\ref{eq:G9}).
\begin{proof}
We refer to \cite{PaSc75} for the fact that very p.d. function $K$
on $X\times X$ can be realized as the covariance kernel of a Gaussian
process in some probability space. Conversely, to stress the idea,
we include a proof that K, as defined in (\ref{eq:K1}), is positive
definite: Let $\left\{ c_{i}\right\} _{i=1}^{n}\subset\mathbb{C}$
and $\left\{ x_{i}\right\} _{i=1}^{n}\subset X$, then we have
\[
\sum\nolimits _{i}\sum\nolimits _{j}\overline{c_{i}}c_{j}K\left(x_{i},x_{j}\right)=\mathbb{E}\left[\big|\sum c_{i}W_{x_{i}}\big|^{2}\right]\geq0,
\]
i.e., $K$ is p.d. 
\end{proof}
Let $\left(X,\mathscr{B},\nu\right)$ be a $\sigma$-finite measure
space, and let $\mathscr{B}_{fin}=\left\{ E\in\mathscr{B}:\nu\left(E\right)<\infty\right\} $.
Below we consider the following kernel $K$ on $\mathscr{B}_{fin}\times\mathscr{B}_{fin}$:
Set 
\begin{equation}
K\left(A,B\right)=\nu\left(A\cap B\right),\quad A,B\in\mathscr{B}_{fin}\label{eq:n4}
\end{equation}
and let $\mathscr{H}_{K^{\left(\nu\right)}}$ denote the associated
RKHS. 

This particular classes of Gaussian processes and RKHSs play an important
role in our approach to manifold learning, and to approximation via
Monte Carlo simulations and associated Karhunen--Loève expansions.
We also note that the resulting stochastic analysis is of interest
for the same purpose. It is further outlined in the remaining two
subsections below.
\begin{prop}
~
\begin{enumerate}
\item \label{enu:ad1}$K=K^{\left(\nu\right)}$ in (\ref{eq:n4}) is positive
definite.
\item \label{enu:ad2} $K^{\left(\nu\right)}$ is the covariance kernel
for the stationary Wiener process $W=W^{\left(\nu\right)}$ indexed
by $\mathscr{B}_{fin}$, i.e., Gaussian, mean zero, and 
\begin{equation}
\mathbb{E}\left(W_{A}W_{B}\right)=K^{\left(\nu\right)}\left(A,B\right)=\nu\left(A\cap B\right).\label{eq:n5}
\end{equation}
\item \label{enu:ad3} If $f\in L^{2}\left(\nu\right)$, and $W_{f}=\int_{X}f\left(x\right)dW_{x}$
denotes the corresponding Wiener integral, then 
\[
\mathbb{E}\left(\left|W_{f}\right|^{2}\right)=\int_{X}\left|f\right|^{2}d\nu;
\]
in particular, if $f=\sum_{i}\alpha_{i}\chi_{A_{i}}$, then 
\[
\sum\nolimits _{i}\sum\nolimits _{j}\alpha_{i}\alpha_{j}K^{\left(\nu\right)}\left(A_{i},A_{j}\right)=\int_{X}\left|\sum\nolimits _{i}\alpha_{i}\chi_{A_{i}}\right|^{2}d\nu.
\]
\item \label{enu:ad4}The RKHS $\mathscr{H}_{K^{\left(\nu\right)}}$ of
the positive definite kernel in (\ref{eq:n4}) consists of functions
$F$ on $\mathscr{B}_{fin}$ represented by $f\in L^{2}\left(\nu\right)$
via 
\begin{equation}
F\left(A\right)=F_{f}\left(A\right)=\int_{A}fd\nu,\quad A\in\mathscr{B}_{fin};\label{eq:n5b}
\end{equation}
and 
\begin{equation}
\left\Vert F_{f}\right\Vert _{\mathscr{H}_{K^{\left(\nu\right)}}}^{2}=\left\Vert f\right\Vert _{L^{2}\left(\nu\right)}^{2}=\int_{X}\left|f\right|^{2}d\nu.\label{eq:n6}
\end{equation}
\item \label{enu:ad5}TThe map specified by 
\begin{equation}
\Psi\left(K^{\left(\nu\right)}\left(\cdot,A\right)\right)=\Psi\left(\nu\left(\left(\cdot\right)\cap A\right)\right)=\chi_{A},\quad\forall A\in\mathscr{B}_{fin}\label{eq:7}
\end{equation}
extends by linearity and by limits to an isometry 
\begin{equation}
\Psi:\mathscr{H}\left(K^{\left(\nu\right)}\right)\longrightarrow L^{2}\left(\nu\right).\label{eq:q8}
\end{equation}
More generally if $F_{f}\in\mathscr{H}\left(K^{\left(\nu\right)}\right)$
is as in (\ref{eq:n5b}), then $\Psi\left(F_{f}\right)=f\in L^{2}\left(\nu\right)$. 
\end{enumerate}
\end{prop}

\begin{proof}
(\ref{enu:ad1}) One checks that
\[
\sum\nolimits _{i}\sum\nolimits _{j}\alpha_{i}\alpha_{j}K^{\left(\nu\right)}\left(A_{i},A_{j}\right)=\int\left|\sum\nolimits _{i}\alpha_{i}1_{A_{i}}\right|^{2}d\nu\geq0,
\]
which holds for all $\left\{ \alpha_{i}\right\} _{1}^{n}$ and $\left\{ A_{i}\right\} _{1}^{n}$
with $\alpha_{i}\in\mathbb{R}$, $A_{i}\in\mathscr{B}_{fin}$, and
for all $n\in\mathbb{N}$. Here we assume that $\alpha_{i}\in\mathbb{R}$.
Note that a real symmetric matrix $T$ is p.d. if and only if $\langle Tx,x\rangle\geq0$,
for all real vectors $x$. (cf N.N. Vakhania, V.I. Tarieladze, S.A.
Chobayan, Probability distributions on Banach spaces (1987) \cite{MR1435288},
Kluwer; the proof of Theorem 1.1, page 128 in Russian original edition).

(\ref{enu:ad2}) Follows from \thmref{kol}. 

(\ref{enu:ad3}) One first verifies that 
\[
\sum\nolimits _{i}\sum\nolimits _{j}\alpha_{i}\alpha_{j}K^{\left(\nu\right)}\left(A_{i},A_{j}\right)=\int_{X}\left|\sum\nolimits _{i}\alpha_{i}\chi_{A_{i}}\right|^{2}d\nu.
\]
Then use density and standard approximation by simple functions to
get the desired conclusion.

(\ref{enu:ad4}) Note that every $F\in\mathscr{H}(K^{\left(\nu\right)})$
is a $\sigma$-additive signed measure, and $dF\ll d\nu$. Moreover,
for $A,B\in\mathscr{B}_{fin}$, one has
\[
K^{\left(\nu\right)}\left(A,B\right)=\int_{B}1_{A}\left(x\right)d\nu\left(x\right)=\mu\left(B\cap A\right),
\]
so that 
\begin{equation}
\frac{dK^{\left(\nu\right)}\left(\cdot,A\right)}{d\nu}=1_{A}.\label{eq:n1-5}
\end{equation}
Now, for all $B\in\mathscr{B}_{fin}$, we have 
\begin{equation}
F\left(B\right)=\int_{B}\left(\frac{dF}{d\nu}\right)d\nu.\label{eq:N11}
\end{equation}
The function $F$ on $\mathscr{B}_{fin}$ is in $\mathscr{H}(K^{\left(\nu\right)})$
if and only if the following supremum, on all positions $\left\{ A_{i}\right\} _{1}^{n},$$A_{i}\in\mathscr{B}_{fin}$,
$A_{i}\cap A_{j}=\emptyset$, for $i\neq j$, is finite with the upper
bound depending on $F$ 

\[
\sup_{{A_{i}}}\sum_{i}\frac{|F(A_{i})|^{2}}{\mu(A_{i})}<\infty.
\]
Indeed, from the reproducing property, we get that 
\begin{eqnarray*}
F\left(B\right) & = & \left\langle K\left(B,\cdot\right),F\right\rangle _{\mathscr{H}(K^{\left(\nu\right)})}\\
 & = & \int\frac{dK\left(B,\cdot\right)}{d\nu}\left(\frac{dF}{d\nu}\right)d\nu\\
 & = & \int\chi_{B}\left(\frac{dF}{d\nu}\right)d\nu=\int_{B}\left(\frac{dF}{d\nu}\right)d\nu.
\end{eqnarray*}
Therefore, if $F_{i}\in\mathscr{H}(K^{\left(\nu\right)})$, $i=1,2$,
then 
\[
\left\langle F_{1},F_{2}\right\rangle _{\mathscr{H}(K^{\left(\nu\right)})}=\int\left(\frac{dF_{1}}{d\nu}\right)\left(\frac{dF_{2}}{d\nu}\right)d\nu.
\]

A more detailed discussion of this part of stochastic calculus can
be found at various places in the literature; see e.g., \cite{MR2966130,MR3800275,MR3888850,MR4020693}. 

(\ref{enu:ad5}) This follows from the fact that (i) $\Psi$ is isometric
by its very definition; (ii) the linear spans, $\text{span}\left\{ K^{(\nu)}(\cdot,A):A\in\mathscr{B}_{fin}\right\} $,
and $\text{span}\left\{ \chi_{A}:A\in\mathscr{B}_{fin}\right\} $
are dense in the respective Hilbert spaces, and (iii) an application
of standard approximation by simple functions in $L^{2}$-spaces.
\end{proof}
\begin{prop}
\label{prop:Qn}Let $\mathscr{B}_{fin}\times\mathscr{B}_{fin}\xrightarrow{\;K\;}\mathbb{R}$
be the p.d. kernel as in (\ref{eq:n4}), and $\mathscr{H}_{K^{\left(\nu\right)}}$
the RKHS of $K$. Suppose $\left\{ Q_{n}\right\} _{\mathbb{N}_{0}}$
is an effective system in $\mathscr{H}_{K^{\left(\nu\right)}}$, and
let $\Psi$ be the isometry specified in (\ref{eq:7})--(\ref{eq:q8}).
Then 
\[
\left\{ \Psi Q_{n}\right\} _{n\in\mathbb{N}_{0}}
\]
is effective in the closed subspace 
\[
\Psi\left(\mathscr{H}_{K^{\left(\nu\right)}}\right)\subset L^{2}\left(\Omega,\mathbb{P}\right).
\]
\end{prop}

\begin{proof}
For all $F,G\in\mathscr{H}_{K^{\left(\nu\right)}}$, we have 
\begin{align*}
\mathbb{E}\left[\Psi\left(F\right)\Psi\left(G\right)\right] & =\left\langle F,G\right\rangle _{\mathscr{H}_{K^{\left(\nu\right)}}}\\
 & =\sum_{n\in\mathbb{N}_{0}}\left\langle Q_{n}F,Q_{n}G\right\rangle _{\mathscr{H}_{K^{\left(\nu\right)}}}\\
 & =\sum_{n\in\mathbb{N}_{0}}\mathbb{E}\left[\Psi\left(Q_{n}F\right)\Psi\left(Q_{n}G\right)\right].
\end{align*}
\end{proof}

\subsection{\label{subsec:Stochastic-analysis-and}Stochastic analysis and p.d.
kernels}

Consider a suitable transform for the more general class of positive
definite function $K:S\times S\rightarrow\mathbb{R}$ (or $\mathbb{C}$).
We can introduce a $\sigma$-algebra on $S$ generated by $\left\{ K\left(\cdot,t\right)\right\} _{t\in S}$
such that $\mu\mapsto\int\mu\left(ds\right)K\left(s,\cdot\right)$
makes sense, where $\mu$ is a suitable signed measure (i.e., $\mu K\mu<\infty$;
see (\ref{def:m2})); but it is more useful to consider linear functionals
$l$ on $\mathscr{H}\left(K\right)$.
\begin{defn}
\label{def:L2}We say that $l\in\mathscr{L}_{2}\left(K\right)$ if,
there exists $C=C_{l}<\infty$ such that $\left|l\left(G\right)\right|\leq C_{l}\left\Vert G\right\Vert _{\mathscr{H}\left(K\right)}$,
for all $G\in\mathscr{H}\left(K\right)$. On $\mathscr{L}_{2}\left(K\right)$,
introduce the Hilbert inner product
\[
lKl:=\left\langle l,l\right\rangle _{\mathscr{L}_{2}}=l\left(\text{acting in \ensuremath{s}}\right)K\left(s,t\right)l\left(\text{acting in \ensuremath{t}}\right).
\]
\end{defn}

\begin{defn}
\label{def:m2}Given $K:S\times S\rightarrow\mathbb{R}$ p.d., let
$\mathscr{B}_{\left(S,K\right)}$ be the cylinder $\sigma$-algebra
on $S$ and consider signed measures $\mu$ on $\mathscr{B}_{\left(S,K\right)}$.
More precisely, $\mathscr{B}_{\left(S,K\right)}$ is the weakest $\sigma$-algebra
with respect to which $K$ is measurable. Let 
\[
\mathfrak{M}_{2}\left(K\right):=\left\{ \mu\mid\mu K\mu:=\left\langle \mu,\mu\right\rangle _{\mathfrak{M}_{2}}=\iint\mu\left(ds\right)K\left(s,t\right)\mu\left(dt\right)<\infty\right\} 
\]
where $\left\langle \mu,\mu\right\rangle _{\mathfrak{M}_{2}}$ is
a Hilbert pre-inner. 

The general idea in Definition \ref{def:m2} here is to make use of
a fixed p.d. kernel $K$ on $S\times S$, and then from this introduce
a new pre-Hilbert inner product on signed measures as follows. The
construction will be in three steps: (i) we use $K$ to create a metric
on $S$, and we then consider the corresponding Borel $\sigma$-algebra;
(ii) and we may therefore introduce the real vector space of all the
corresponding signed Borel measures $\mu$, $\nu$ etc. (iii) In our
definition of the $K$-pre-Hilbert inner product we make it precise
as a pre-Hilbert inner product via the formula $\mu K\nu$ where $K$
now takes the form of an integral kernel. 
\end{defn}

The basic idea with the approach using linear functionals, $l:\mathscr{H}\left(K\right)\rightarrow\mathbb{R}$,
is to extend from the case of finite atomic measures to a precise
completion. 
\begin{example}
For $\left\{ c_{i}\right\} _{1}^{n}$, $\left\{ s_{i}\right\} _{1}^{n}$
$c_{i}\in\mathbb{R}$, let 
\begin{equation}
l\left(G\right):=\sum_{i}c_{i}G\left(s_{i}\right),\;\forall G\in\mathscr{H}\left(K\right);
\end{equation}
which is the case when $\mu=\sum_{i}c_{i}\delta_{s_{i}}$, with $\delta_{s_{i}}$
denoting the Dirac measure. Setting 
\[
T_{K}\left(\mu\right)=\sum_{i}c_{i}K\left(s_{i},\cdot\right)\in\mathscr{H}\left(K\right),
\]
we then get 
\begin{equation}
l\left(G\right)=\left\langle T_{K}\left(\mu\right),G\right\rangle _{\mathscr{H}\left(K\right)},\;\forall G\in\mathscr{H}\left(K\right).
\end{equation}
Moreover, the following isometric property holds for all finite atomic
measures $\mu$: 
\[
\left\Vert T_{K}\left(\mu\right)\right\Vert _{\mathscr{H}\left(K\right)}=\left\Vert \mu\right\Vert _{\mathfrak{M}_{2}\left(K\right)}.
\]
\end{example}

We include some basic facts about $\mathfrak{M}_{2}\left(K\right)$
and $\mathscr{L}_{2}\left(K\right)$. 
\begin{thm}
\label{thm:kernel}Suppose the p.d. function $K:X\times X\rightarrow\mathbb{R}$
has the representation
\begin{equation}
K\left(s,t\right)=\sum_{n\in\mathbb{N}}e_{n}\left(s\right)e_{n}\left(t\right)\label{eq:t7}
\end{equation}
(pointwise convergence) for some functions $\left(e_{n}\right)_{n\in\mathbb{N}}$.
\begin{enumerate}
\item Then $\left\{ e_{n}\right\} \subset\mathscr{H}\left(K\right)$, and
for all $G\in\mathscr{H}\left(K\right)$, 
\[
\left\Vert G\right\Vert _{\mathscr{H}\left(K\right)}^{2}=\sum\left|\left\langle G,e_{n}\right\rangle _{\mathscr{H}\left(K\right)}\right|^{2}.
\]
That is, $\left(e_{n}\right)$ is a Parseval frame in $\mathscr{H}\left(K\right)$. 
\item We have 
\[
l\in\mathscr{L}_{2}\left(K\right)\Longleftrightarrow\sum_{n\in\mathbb{N}}\left|l\left(e_{n}\right)\right|^{2}<\infty,
\]
i.e., $l\in\mathscr{L}_{2}\left(K\right)\Longleftrightarrow\left(l\left(e_{n}\right)\right)_{n\in\mathbb{N}}\in l^{2}\left(\mathbb{N}\right)$.
\end{enumerate}
\end{thm}

\begin{proof}
~

Part (1). It suffices to check that, for all $N$, $\left(c_{i}\right)_{1}^{N}\subset\mathbb{R}$,
and all $\left(x_{i}\right)_{1}^{N}\subset X$, one has 
\[
\left|\sum_{i}c_{i}e_{n_{0}}\left(x_{i}\right)\right|^{2}\leq\sum_{i}\sum_{j}c_{i}c_{j}K\left(x_{i},x_{j}\right),\;\forall n_{0}\in\mathbb{N}.
\]
Indeed, 
\begin{align*}
0\leq\left|\sum_{i}c_{i}e_{n_{0}}\left(x_{i}\right)\right|^{2} & =\sum_{i}\sum_{j}c_{i}c_{j}e_{n_{0}}\left(x_{i}\right)e_{n_{0}}\left(x_{j}\right)\\
 & \leq\sum_{i}\sum_{j}c_{i}c_{j}\sum_{n\in\mathbb{N}}e_{n}\left(x_{i}\right)e_{n}\left(x_{j}\right)\\
 & =\sum_{i}\sum_{j}c_{i}c_{j}K\left(x_{i},x_{j}\right).
\end{align*}
This implies that $e_{n_{0}}\in\mathscr{H}\left(K\right)$, for all
$n_{0}\in\mathbb{N}$. Now, for all $G=\sum_{\text{finite}}c_{i}K\left(\cdot,x_{i}\right)\in\mathscr{H}\left(K\right)$,
it follows that 
\begin{align*}
\sum_{n\in\mathbb{N}}\left|\left\langle G,e_{n}\right\rangle _{\mathscr{H}\left(K\right)}\right|^{2} & =\left\Vert G\right\Vert _{\mathscr{H}\left(K\right)}^{2}=\sum c_{i}c_{j}K\left(x_{i},x_{j}\right),
\end{align*}
which extends by density and linearity to all $G\in\mathscr{H}\left(K\right)$. 

Part (2). Given (\ref{eq:t7}), if $G\in\mathscr{H}\left(K\right)$,
then 
\begin{equation}
G\left(t\right)=\sum\left\langle G,e_{n}\right\rangle _{\mathscr{H}\left(K\right)}e_{n}\left(t\right),\label{eq:t8}
\end{equation}
(norm convergence) and 
\begin{align*}
\left|l\left(G\right)\right|^{2} & =\left|\sum\left\langle G,e_{n}\right\rangle _{\mathscr{H}\left(K\right)}l\left(e_{n}\right)\right|^{2}\\
 & \leq\sum\left|\left\langle G,e_{n}\right\rangle _{\mathscr{H}\left(K\right)}\right|^{2}\sum\left|l\left(e_{n}\right)\right|^{2}\\
 & \leq C_{l}\left\Vert G\right\Vert _{\mathscr{H}\left(K\right)}^{2}
\end{align*}
using the Cauchy-Schwarz inequality. 

Define $V:\mathscr{H}\left(K\right)\rightarrow l^{2}\left(\mathbb{N}\right)$
by 
\[
V\left(G\right)=\left(\left\langle G,e_{n}\right\rangle _{\mathscr{H}\left(K\right)}\right)_{n\in\mathbb{N}}.
\]
Then $V$ is isometric by (1). Moreover, the adjoint $V^{*}:l^{2}\left(\mathbb{N}\right)\rightarrow\mathscr{H}\left(K\right)$
is given by 
\[
V^{*}\left(\left(a_{n}\right)\right)\left(t\right)=\sum a_{n}e_{n}\left(t\right)\in\mathscr{H}\left(K\right),
\]
so that $V^{*}V=I_{\mathscr{H}\left(K\right)}$. Note that (\ref{eq:t8})
is a canonical representation, but it is not unique unless $\left(e_{n}\right)$
is assumed to be an ONB and not just a Parseval frame.

More details: Suppose $l$ is a linear functional on $\mathscr{H}\left(K\right)$
and $\left(l\left(e_{n}\right)\right)\in l^{2}$, then 
\[
F=F_{l}=\sum l\left(e_{n}\right)e_{n}\in\mathscr{H}\left(K\right)
\]
is well defined, since $\left\{ e_{n}\right\} $ is a Parseval frame.
We also have 
\begin{align*}
l\left(G\right) & =\sum\left\langle G,e_{n}\right\rangle _{\mathscr{H}\left(K\right)}l\left(e_{n}\right)\\
 & =\left\langle \sum l\left(e_{n}\right)e_{n},G\right\rangle _{\mathscr{H}\left(K\right)}=\left\langle F_{l},G\right\rangle _{\mathscr{H}\left(K\right)},
\end{align*}
which is the desired conclusion, i.e., $T_{K}\left(l\right)=F_{l}$
and $T_{K}^{*}\left(F_{l}\right)=l$.
\end{proof}
\begin{example}
Let $S=\left[0,1\right]$, $K\left(s,t\right)=s\wedge t$. This is
a \emph{relative kernel}, since we specify 
\[
\mathscr{H}\left(K\right)=\left\{ F:F'\in L^{2}\left(0,1\right),F\left(0\right)=0\right\} ,
\]
with $\left\Vert F\right\Vert _{\mathscr{H}\left(K\right)}^{2}=\int_{0}^{1}\left|F'\left(s\right)\right|^{2}ds$,
and we then get
\begin{align*}
\left\langle F\left(\cdot\right),\cdot\wedge t\right\rangle _{\mathscr{H}\left(K\right)} & =\int_{0}^{1}F'\left(s\right)\chi_{\left[0,t\right]}\left(s\right)ds\\
 & =\int_{0}^{t}F'\left(s\right)ds\\
 & =F\left(t\right)-F\left(0\right)=F\left(t\right).
\end{align*}
In this example, 
\begin{align*}
\left(T_{K}\mu\right)\left(t\right) & =\int\mu\left(ds\right)s\wedge t=t-\frac{1}{2}t^{2},\\
\mu K\mu & =\left\Vert T_{K}\mu\right\Vert _{\mathscr{H}\left(K\right)}^{2}=\int_{0}^{1}\left|\left(T_{K}\mu\right)'\left(t\right)\right|^{2}dt=\frac{1}{3},
\end{align*}
where $\left(T_{K}\mu\right)'\left(t\right)=\mu\left(\left[t,1\right]\right)$.

The example below serves to illustrate that the functionals from Definition
\ref{def:m2} and Theorem \ref{thm:kernel} might in fact be Schwartz
distributions.
\end{example}

\begin{example}
\label{exa:Kst1} Let $K\left(s,t\right)=\frac{1}{1-st}$, $s,t\in\left(-1,1\right)$,
and let $\mathscr{H}\left(K\right)$ be the corresponding RKHS. The
kernel $K$ is p.d. since 
\[
K\left(s,t\right)=\sum_{0}^{\infty}s^{n}t^{n}=\sum_{0}^{\infty}e_{n}\left(s\right)e_{n}\left(t\right),
\]
where $e_{n}\left(t\right)=t^{n}$, and $\left\{ e_{n}\right\} $
is an ONB in $\mathscr{H}\left(K\right)$. Now fix $n>0$ and consider
the linear functional
\begin{equation}
l\left(G\right)=\left\langle G,e_{n}\left(\cdot\right)\right\rangle _{\mathscr{H}\left(K\right)}.\label{eq:t4}
\end{equation}
Then $\left|l\left(G\right)\right|\leq\left\Vert G\right\Vert _{\mathscr{H}\left(K\right)}\left\Vert e_{n}\right\Vert _{\mathscr{H}\left(K\right)}\leq\left\Vert G\right\Vert _{\mathscr{H}\left(K\right)}$,
by Schwarz and the fact that $\left\Vert e_{n}\right\Vert _{\mathscr{H}\left(K\right)}=1$.
It follows that $C_{l}=1$. 

However, the LHS of (\ref{eq:t4}) cannot be realized by a measure
on $\left(-1,1\right)$. In fact, the unique solution is given by
$l=\frac{1}{n!}\delta_{0}^{\left(n\right)}\in\mathscr{L}_{2}\left(K\right)$,
where $\delta_{0}^{\left(n\right)}$ is the Schwartz distribution
(not a measure of course): $\delta_{0}^{\left(n\right)}\left(\varphi\right)=\left(-1\right)^{n}\varphi^{\left(n\right)}\left(0\right)$,
$\varphi^{\left(n\right)}=\left(\frac{d}{dt}\right)^{n}\varphi$,
for all $\varphi\in C_{c}\left(-1,1\right)$. 

In this example, $\delta_{0}^{\left(n\right)}\in\mathscr{L}_{2}\left(K\right)$,
but \emph{not} in $\mathfrak{M}_{2}\left(K\right)$. However, it is
in the closure of $\mathfrak{M}_{2}\left(K\right)$. The plan is therefore
to extend from the space of $\mu=\sum_{i}c_{i}\delta_{s_{i}}$ to
the more general measures, and (in some examples) to the case of Schwartz
distributions.

Moreover, we have 
\begin{equation}
\delta_{x_{0}}=\sum_{n\in\mathbb{N}_{0}}\frac{x_{0}^{n}}{n!}\delta_{0}^{\left(n\right)},\;\left|x_{0}\right|<1.\label{eq:t12}
\end{equation}
One may apply $T_{K}$ to both sides of (\ref{eq:t12}), and get 
\begin{align*}
\text{LHS}_{\left(\ref{eq:t12}\right)} & =T_{K}\left(\delta_{x_{0}}\right)\left(t\right)=\frac{1}{1-x_{0}t}=K\left(x_{0},t\right),\\
\text{RHS}_{\left(\ref{eq:t12}\right)} & =\sum_{n\in\mathbb{N}_{0}}\frac{x_{0}^{n}}{n!}\left(T_{K}\left(\delta_{0}^{\left(n\right)}\right)\right)\left(t\right)=\sum_{n\in\mathbb{N}_{0}}\frac{x_{0}^{n}}{n!}n!e_{n}\left(t\right)=\sum_{n\in\mathbb{N}_{0}}x_{0}^{n}t^{n}=K\left(x_{0},t\right).
\end{align*}
Note that (\ref{eq:t12}) lives in $\mathscr{L}_{2}\left(K\right)$
and so depends on $K$. We will later consider the case $K\left(s,t\right)=e^{st}$. 
\end{example}

\subsection{\label{subsec:The-isomorphism}The isomorphism $T_{K}:\mathscr{L}_{2}\left(K\right)\rightarrow\mathscr{H}\left(K\right)$}

\textbf{Motivation.} The RKHS $\mathscr{H}\left(K\right)$ is defined
abstractly wheres $\mathfrak{M}_{2}\left(K\right)$ and $\mathscr{L}_{2}\left(K\right)$
are explicit (see Def. \ref{def:L2} \& \ref{def:m2}), and more ``natural''.
Starting with a p.d. kernel $K$ on $S$, then the goal is to identify
Hilbert spaces of measures $\mathfrak{M}_{2}\left(K\right)$, or Hilbert
spaces of Schwartz distributions, say $\mathscr{L}_{2}\left(K\right)$,
which can give a more explicit norm, with a natural isometry $\mathscr{L}_{2}\left(K\right)\xrightarrow{T_{K}}\mathscr{H}\left(K\right)$.

The mapping $T_{K}$ may be defined for measures $\mu$, viewed as
elements in $\mathscr{L}_{2}\left(K\right)$, as the generalized integral
operator: 
\begin{defn}
For $\mu\in\mathfrak{M}_{2}\left(K\right)$, set 
\end{defn}

\begin{equation}
\left(T_{K}\mu\right)\left(t\right)=\int\mu\left(ds\right)K\left(s,t\right),\;\mu\in\mathfrak{M}_{2}\left(K\right).
\end{equation}

For $l\in\mathscr{L}_{2}\left(K\right)$, by Riesz there exists a
unique $F=F_{l}\in\mathscr{H}\left(K\right)$ such that $l\left(G\right)=\left\langle F_{l},G\right\rangle _{\mathscr{H}\left(K\right)}$,
for all $G\in\mathscr{H}\left(K\right)$. 
\begin{defn}
For $l\in\mathscr{L}_{2}\left(K\right)$, set 
\begin{equation}
T_{K}\left(l\right)=F_{l},\;\text{and so }T_{K}^{*}\left(F_{l}\right)=l.
\end{equation}
Alternatively, one may define 
\begin{equation}
F_{l}:=T_{K}\left(l\right)=l\left(\text{acting in \ensuremath{s}}\right)\left(K\left(s,\cdot\right)\right),
\end{equation}
by the reproducing property.
\end{defn}

\begin{lem}
Let $K$ be p.d. on $S\times S$ with the cylinder $\sigma$-algebra
$\mathscr{B}_{\left(S,K\right)}$. Let $\mu$ be a positive measure
on $\mathscr{B}_{\left(S,K\right)}$ such that $\int K\left(s,s\right)\mu\left(ds\right)<\infty$,
then $L^{2}\left(\mu\right)\hookrightarrow\mathscr{H}\left(K\right)$
is a bounded embedding, i.e., $fd\mu\in\mathfrak{M}_{2}\left(K\right)$,
for all $f\in L^{2}\left(\mu\right)$. 
\end{lem}

\begin{proof}
Given $f\in L^{2}\left(\mu\right)$, then 
\begin{eqnarray*}
\left\Vert \int f\left(s\right)K\left(s,\cdot\right)\mu\left(ds\right)\right\Vert _{\mathscr{H}\left(K\right)} & \leq & \int\left|f\left(s\right)\right|\left\Vert K\left(s,\cdot\right)\right\Vert _{\mathscr{H}\left(K\right)}\mu\left(ds\right)\\
 & \underset{\mathclap{\left(\text{Schwarz}\right)}}{\leq} & \left\Vert f\right\Vert _{L^{2}\left(\mu\right)}\left(\int K\left(s,s\right)\mu\left(ds\right)\right)^{1/2}\\
 & = & C_{\mu}^{1/2}\left\Vert f\right\Vert _{L^{2}\left(\mu\right)}.
\end{eqnarray*}
Hence, $\int K\left(s,\cdot\right)f\left(s\right)\mu\left(ds\right)=T_{K}\left(fd\mu\right)\in\mathscr{H}\left(K\right)$
and 
\begin{align*}
\left\Vert T_{K}\left(fd\mu\right)\right\Vert _{\mathscr{H}\left(K\right)}^{2} & =\iint f\left(s\right)K\left(s,t\right)f\left(t\right)\mu\left(ds\right)\mu\left(dt\right)\\
 & =\left(fd\mu\right)K\left(fd\mu\right)\\
 & =\left\Vert fd\mu\right\Vert _{\mathfrak{M}_{2}\left(K\right)}^{2}.
\end{align*}
\end{proof}
\begin{lem}
If $\mu\in\mathfrak{M}_{2}\left(K\right)$, then 
\[
\left(T_{K}\mu\right)\left(t\right)=\int\mu\left(ds\right)K\left(s,t\right)
\]
is well defined, and $\left(T_{K}\mu\right)\left(\cdot\right)\in\mathscr{H}\left(K\right)$.
Moreover, $T_{K}:\mathfrak{M}_{2}\left(K\right)\rightarrow\mathscr{H}\left(K\right)$
is isometric, where 
\begin{equation}
\left\Vert T_{K}\mu\right\Vert _{\mathscr{H}\left(K\right)}^{2}=\left\Vert \mu\right\Vert _{\mathfrak{M}_{2}\left(K\right)}^{2}=\iint\mu\left(ds\right)K\left(s,t\right)\mu\left(dt\right).\label{eq:t5}
\end{equation}
\end{lem}

\begin{proof}
With the assumption $\mu\in\mathfrak{M}_{2}\left(K\right)$, the exchange
of $\left\langle \cdot,\cdot\right\rangle _{\mathscr{H}\left(K\right)}$
and integral with respect to $\mu$ can be justified. Thus, 
\begin{align*}
\text{LHS}_{\left(\ref{eq:t5}\right)} & =\left\langle T_{K}\mu,T_{K}\mu\right\rangle _{\mathscr{H}\left(K\right)}\\
 & =\int\left(T_{K}\mu\right)\left(t\right)\mu\left(dt\right)\\
 & =\iint\mu\left(ds\right)K\left(s,t\right)\mu\left(dt\right)\\
 & =\mu K\mu=\left\Vert \mu\right\Vert _{\mathfrak{M}_{2}\left(K\right)}^{2}=\text{RHS}_{\left(\ref{eq:t5}\right)}.
\end{align*}
\end{proof}
\begin{defn}
By a partition $\left\{ E_{i}\right\} $ of $S$, we mean $E_{i}\in\mathscr{B}_{\left(S,K\right)}$
of finite measure, $E_{i}\cap E_{j}=\emptyset$ if $i\neq j$, and
$\cup E_{i}=S$. 
\end{defn}

\begin{lem}
Let $\mu$ be a signed measure, then $\mu\in\mathfrak{M}_{2}\left(K\right)$
if and only if 
\[
\left\Vert \sum_{i}\mu\left(E_{i}\right)K\left(s_{i},\cdot\right)\right\Vert _{\mathscr{H}\left(K\right)}^{2}=\sup_{\text{all partitions}}\sum_{i}\sum_{j}\mu\left(E_{i}\right)K\left(s_{i},s_{j}\right)\mu\left(E_{j}\right)<\infty
\]
with $s_{i}\in E_{i}$. In this case, we have the following norm limit
\[
\lim\left\Vert T_{K}\mu-\sum_{i}\mu\left(E_{i}\right)K\left(s_{i},\cdot\right)\right\Vert _{\mathscr{H}\left(K\right)}=0,
\]
where the limit is taken over filters of partitions. 
\end{lem}

\begin{proof}
For $\mu\in\mathfrak{M}_{2}\left(K\right)$, we must make precise
$\int G\left(s\right)\mu\left(ds\right)$ as a well defined integral.
Again, we take limits over all partitions $\mathscr{P}$ of $S$:
\begin{align*}
\sum_{i}G\left(s_{i}\right)\mu\left(E_{i}\right) & =\sum\mu\left(E_{i}\right)\left\langle K\left(s_{i},\cdot\right),G\left(\cdot\right)\right\rangle _{\mathscr{H}\left(K\right)}\\
 & =\left\langle \sum\mu\left(E_{i}\right)K\left(s_{i},\cdot\right),G\left(\cdot\right)\right\rangle _{\mathscr{H}\left(K\right)}\rightarrow\left\langle T_{K}\mu,G\right\rangle _{\mathscr{H}\left(K\right)},
\end{align*}
and so $T_{K}\mu$ is the $\left\Vert \cdot\right\Vert _{\mathscr{H}\left(K\right)}$-norm
limit of $\sum\mu\left(E_{i}\right)K\left(s_{i},\cdot\right)$. 

More specifically, given a partition $P=\left\{ E_{i}\right\} \in\mathscr{P}$,
let 
\[
T\left(P\right):=\sum_{i}\mu\left(E_{i}\right)K\left(s_{i},\cdot\right)\in\mathscr{H}\left(K\right),
\]
then we have 
\[
\left\Vert T\left(P\right)\right\Vert _{\mathscr{H}\left(K\right)}^{2}=\sum_{i}\sum_{j}\mu\left(E_{i}\right)K\left(s_{i},s_{j}\right)\mu\left(E_{j}\right)
\]
and $T\left(P\right)\rightarrow T_{K}\left(\mu\right)$, as a filter
in partitions. 

In the following two lemmas by a standard application of Fubini's
theorem, one can justify the exchange of the integrals and summations
inside and outsite the inner product\textcolor{blue}{. }
\end{proof}
\begin{lem}
For all $G\in\mathscr{H}\left(K\right)$ and $\mu\in\mathfrak{M}_{2}\left(K\right)$,
we have:
\begin{equation}
\int G\left(t\right)\mu\left(dt\right)=\left\langle G,T_{K}\mu\right\rangle _{\mathscr{H}\left(K\right)}.\label{eq:t6}
\end{equation}
\end{lem}

\begin{proof}
~
\begin{align*}
\text{LHS}_{\left(\ref{eq:t6}\right)} & =\int\left\langle G\left(\cdot\right),K\left(\cdot,t\right)\right\rangle _{\mathscr{H}\left(K\right)}\mu\left(dt\right)\\
 & =\left\langle G,\int\mu\left(dt\right)K\left(\cdot,t\right)\right\rangle _{\mathscr{H}\left(K\right)}\\
 & =\left\langle G,T_{K}\mu\right\rangle _{\mathscr{H}\left(K\right)}.
\end{align*}
We conclude that $\int G\left(t\right)\mu\left(dt\right)=\left\langle T_{K}\mu,G\right\rangle _{\mathscr{H}\left(K\right)}\in\mathbb{\mathbb{R}}$.
Note that since $\mu\in\mathfrak{M}_{2}\left(K\right)$, we may exchange
integral with $\left\langle \cdot,\cdot\right\rangle _{\mathscr{H}\left(K\right)}$.
\end{proof}
\begin{lem}
Let $\left(S,\mathscr{B}_{\left(S,K\right)},K\right)$ be as above,
and consider signed measures $\mu$ on $\left(S,\mathscr{B}_{\left(S,K\right)}\right)$.
Then
\[
\mu\in\mathfrak{M}_{2}\left(K\right)\Longleftrightarrow\left|\int Gd\mu\right|\leq C_{\mu}\left\Vert G\right\Vert _{\mathscr{H}\left(K\right)},\;\forall G\in\mathscr{H}\left(K\right).
\]
\end{lem}

\begin{proof}
Given $\mu\in\mathfrak{M}_{2}\left(K\right)$, then 
\begin{align*}
\left|\int G\left(t\right)\mu\left(dt\right)\right| & =\left|\int\left\langle K\left(\cdot,t\right),G\left(\cdot\right)\right\rangle _{\mathscr{H}\left(K\right)}\mu\left(dt\right)\right|\\
 & \leq\left|\left\langle \int K\left(\cdot,t\right)\mu\left(dt\right),G\right\rangle _{\mathscr{H}\left(K\right)}\right|\\
 & =\left|\left\langle T_{K}\mu,G\right\rangle _{\mathscr{H}\left(K\right)}\right|\\
 & \leq\left\Vert T_{K}\mu\right\Vert _{\mathscr{H}\left(K\right)}\left\Vert G\right\Vert _{\mathscr{H}\left(K\right)},
\end{align*}
and so $C_{\mu}=\left\Vert T_{K}\mu\right\Vert _{\mathscr{H}\left(K\right)}<\infty$. 

Conversely, suppose $\left|\int Gd\mu\right|\leq C_{\mu}\left\Vert G\right\Vert _{\mathscr{H}\left(K\right)}$
for all $G\in\mathscr{H}\left(K\right)$. Then there exists a unique
$F\in\mathscr{H}\left(K\right)$ such that 
\[
\int Gd\mu=\left\langle G,F\right\rangle _{\mathscr{H}\left(K\right)},\;\forall G\in\mathscr{H}\left(K\right).
\]
Now take $G=K\left(\cdot,t\right)$, and by the reproducing property,
we get 
\[
\underset{T_{K}\mu}{\underbrace{\int\mu\left(ds\right)K\left(s,t\right)}}=\left\langle K\left(\cdot,t\right),F\right\rangle _{\mathscr{H}\left(K\right)}=F\left(t\right),
\]
where $T_{K}\mu=F\in\mathscr{H}\left(K\right)$ by Riesz, and $\mu=T_{K}^{*}F$.
We conclude that $\mu\in\mathfrak{M}_{2}\left(K\right)$. 
\end{proof}
In summary, the goal is to pass from (1) atomic measures to (2) $\sigma$-finite
signed measures, and then to (3) linear functionals on $\mathscr{H}\left(K\right)$,
continues with respect to the $\left\Vert \cdot\right\Vert _{\mathscr{H}\left(K\right)}$
norm. The respective mappings are specified as follows: 
\begin{align}
T_{K}:\sum\nolimits _{i} & c_{i}\delta_{s_{i}}\longmapsto\sum\nolimits _{i}c_{i}K\left(s_{i},\cdot\right)\in\mathscr{H}\left(K\right);\\
\left(T_{K}\mu\right)\left(t\right) & =\int\mu\left(ds\right)K\left(s,\cdot\right)\in\mathscr{H}\left(K\right),\;\mu\in\mathfrak{M}_{2}\left(K\right);\\
\left(T_{K}l\right)\left(t\right) & =l_{\cdot}\left(K\left(\cdot,t\right)\right)\in\mathscr{H}\left(K\right),\;l\in\mathscr{L}_{2}\left(K\right).
\end{align}

In all cases ($l$ as atomic measures, signed measures, or the general
case of linear functionals), the mapping $T_{K}:\mathscr{L}_{2}\left(K\right)\rightarrow\mathscr{H}\left(K\right)$
is an isometry, i.e., $\left\Vert T_{K}l\right\Vert _{\mathscr{H}\left(K\right)}=\left\Vert l\right\Vert _{\mathscr{L}_{2}\left(K\right)}$.
In general, we have 
\begin{equation}
T_{K}\left(\mathfrak{M}_{2}\left(K\right)\right)\subseteq T_{K}\left(\mathscr{L}_{2}\left(K\right)\right)=\mathscr{H}\left(K\right),
\end{equation}
where $T_{K}\left(\mathfrak{M}_{2}\left(K\right)\right)$ is dense
in $\mathscr{H}\left(K\right)$ in the $\left\Vert \cdot\right\Vert _{\mathscr{H}\left(K\right)}$
norm. In some cases, 
\begin{equation}
T_{K}\left(\mathfrak{M}_{2}\left(K\right)\right)=\mathscr{H}\left(K\right);\label{eq:t16}
\end{equation}
but not always. 
\begin{rem}
$\mathfrak{M}_{2}\left(K\right)$ and factorizations on $K$ specified
by a positive measure $M$ on $X$. Let 
\[
K\left(s,t\right)=\int\psi\left(s,x\right)\psi\left(t,x\right)M\left(dx\right),
\]
where $\psi\left(s,\cdot\right)\in L^{2}\left(M\right)$. Let $\mu$
be a signed measure on $S$. Then
\[
\mu\in\mathfrak{M}_{2}\left(K\right)\Longleftrightarrow T_{\mu}\psi\in L^{2}\left(M\right)
\]
where $\left(T_{\mu}\psi\right)\left(x\right)=\int\mu\left(ds\right)\psi\left(s,x\right)$.
Indeed, using Fubini, one has 
\begin{align*}
\iint\mu\left(ds\right)K\left(s,t\right)\mu\left(dt\right) & =\int\left(\int\mu\left(ds\right)\psi\left(s,x\right)\int\mu\left(dt\right)\psi\left(t,x\right)\right)M\left(dx\right)\\
 & =\int\left|\left(T_{\mu}\psi\right)\left(x\right)\right|^{2}M\left(dx\right)\\
 & =\left\Vert T_{K}\left(\mu\right)\right\Vert _{\mathscr{H}\left(K\right)}^{2}.
\end{align*}
\end{rem}

\begin{lem}
If $K$ is p.d. on $S$ and if the induced metric $d_{K}\left(s,t\right)=\left\Vert K_{s}-K_{t}\right\Vert _{\mathscr{H}\left(K\right)}$
is bounded, then $\mathfrak{M}_{2}\left(K\right)$ is complete, and
so $T_{K}\left(\mathfrak{M}_{2}\left(K\right)\right)=\mathscr{H}\left(K\right)$. 
\end{lem}

\begin{proof}
Please see the general discussion in section 2 and Lemma 2.1 in \cite{MR3843552}
for details. Sketch: With the metric $d_{K}$, we can complete, and
then use Stone-Weierstrass. Indeed, if $l$ is continuous in the $\left\Vert \cdot\right\Vert _{\mathscr{H}\left(K\right)}$
norm, then there exists a signed measure $\mu$ such that $l\left(G\right)=\int G\left(t\right)\mu\left(dt\right)$. 
\end{proof}
An overview of the relation between $\mathfrak{M}_{2}\left(K\right)$
and $\mathscr{L}_{2}\left(K\right)$. 

Let $K:S\times S\rightarrow\mathbb{R}$ be a given p.d. kernel on
$S$, and consider the RKHS $\mathscr{H}\left(K\right)$, as well
as the isometry $T_{K}$ introduced above. We have:
\[
T_{K}\left(\mathfrak{M}_{2}\left(K\right)\right)\underset{\text{possible \ensuremath{\neq}}}{\subseteq}T_{K}\left(\mathscr{L}_{2}\left(K\right)\right)=\mathscr{H}\left(K\right)
\]
and $T_{K}\left(\mathfrak{M}_{2}\left(K\right)\right)$ is dense in
$\mathscr{H}\left(K\right)$ in the $\left\Vert \cdot\right\Vert _{\mathscr{H}\left(K\right)}$
norm. 

We see that $T_{K}\left(\mathfrak{M}_{2}\left(K\right)\right)$ may
not be closed in $\mathscr{H}\left(K\right)$, and so $\mathfrak{M}_{2}\left(K\right)$
may not be complete. However, the $\mathscr{L}_{2}\left(K\right)$
will be complete since it is a Hilbert space, and $T_{K}\left(\mathscr{L}_{2}\left(K\right)\right)=\mathscr{H}\left(K\right)$.
See \lemref{tc} and \thmref{ton}. 
\begin{lem}
\label{lem:tc}$T_{K}\left(\mathscr{L}_{2}\left(K\right)\right)$
is complete. 
\end{lem}

\begin{proof}
Suppose $\left\{ \mu_{n}\right\} $ is Cauchy in $\mathscr{L}_{2}\left(K\right)$,
$\left\Vert \mu_{n}-\mu_{m}\right\Vert _{\mathscr{L}_{2}\left(K\right)}\rightarrow0$,
then $F_{n}=T_{K}\mu_{n}$ satisfy $\left\Vert F_{n}-F_{m}\right\Vert _{\mathscr{H}\left(K\right)}\rightarrow0$,
and so there exists $F\in\mathscr{H}\left(K\right)$ such that $\left\Vert F_{n}-F\right\Vert _{\mathscr{H}\left(K\right)}\rightarrow0$.
Define $l\left(G\right):=\left\langle F,G\right\rangle _{\mathscr{H}\left(K\right)}$,
for all $G\in\mathscr{H}\left(K\right)$. Then, $l\in\mathscr{L}_{2}\left(K\right)$
and $\left\Vert \mu_{n}-l\right\Vert _{\mathscr{L}_{2}\left(K\right)}\rightarrow0$. 
\end{proof}
\begin{thm}
\label{thm:ton}The isometry $T_{K}$ maps $\mathscr{L}_{2}\left(K\right)$
onto $\mathscr{H}\left(K\right)$.
\end{thm}

\begin{proof}
We shall show that $ker\left(T_{K}^{*}\right)=0$. Suppose $T_{K}^{*}\left(G\right)=0$.
Then for all $\mu$, 
\[
0=\left\langle \mu,T_{K}^{*}G\right\rangle _{\mathfrak{M}_{2}\left(K\right)}=\left\langle T_{K}\mu,G\right\rangle _{\mathscr{H}\left(K\right)}.
\]
Now take $\mu=\delta_{t}$, $t\in S$; then 
\[
0=\left\langle T_{K}\left(\delta_{t}\right),G\right\rangle _{\mathscr{H}\left(K\right)}=\left\langle K\left(\cdot,t\right),G\left(\cdot\right)\right\rangle _{\mathscr{H}\left(K\right)}=G\left(t\right).
\]
Thus, $G\left(t\right)=0$ for all $t\in S$, and so $G=0$ in $\mathscr{H}\left(K\right)$.

In \ref{sec:RKHS} above, in the framework of (general) stochastic
analysis, we presented a natural isometric transform. A key part of
this construction entails an identification of the \textquotedblleft right\textquotedblright{}
Hilbert spaces for the purpose. In the next section below, we shall
now apply this to a general class of stationary Gaussian processes.
Moreover, in this stochastic framework, we then arrive at anew and
explicit transform which serves as an infinite-dimensional stochastic
analysis-Fourier transform; see Theorem \ref{thm:For-the-Gaussian}
and Proposition \ref{prop:TXs}.
\end{proof}

\section{\label{sec:infinite}An infinite-dimensional Fourier transform for
Gaussian processes via kernel analysis}

Our last section below deals with a number of applied transform-results.
In more detail, they are infinite-dimensional and stochastic transforms.
They are motivated to a large extent by our results (see \cite{jorgensen2019dimension})
on dynamical PCA (DPCA). Moreover, our results below will combine
two main themes in our paper, i.e., combining (i) adaptive kernel
tricks (positive definite kernels, their Hilbert spaces, and choices
of feature spaces) with (ii) a new analysis of corresponding classes
of Gaussian processes/fields. This approach also builds on our parallel
results on Monte Carlo simulation and use of Karhunen-Loève analysis.
These new transforms, and their applications (see especially \corref{hl}
and the subsequent results in sect. \ref{sec:infinite}), are motivated
directly by DPCA analysis, and they apply directly to design of DPCA
algorithms.

In this section we discuss two of the positive definite kernels $K$
used in \cite{jorgensen2019dimension}. 

We show that each kernel $K$ is associated with a certain transform
for its reproducing kernel Hilbert space (RKHS) $\mathcal{H}\left(K\right)$.
The transform is studied in detail; -- it may be viewed as an infinite-dimensional
Fourier transform, see \defref{gf}. In detail, this transform $\mathcal{T}$
is defined on an $L^{2}$ path-space $L^{2}\left(\Omega,\mathbb{P}\right)$
of Brownian motion; and $\mathcal{T}$ is shown to be an isometric
isomorphism of $L^{2}\left(\Omega,\mathbb{P}\right)$ onto the RKHS
$\mathcal{H}\left(K\right)$; see \corref{hl}. For earlier results
dealing with the use of RKHSs in PCA, and related areas, see e.g.,
\cite{MR2239907,MR3097610,MR3843387}.\\

Consider the following positive definite (p.d.) kernel on $\mathbb{R}\times\mathbb{R}$;
\begin{equation}
K\left(s,t\right):=e^{-\frac{1}{2}\left|s-t\right|},\;s,t\in\mathbb{R}.\label{eq:e1}
\end{equation}
In order to understand its PCA properties, we consider the top part
of the spectrum in sampled versions of (\ref{eq:e1}). We show below
that $K$ is the covariance kernel of the complex process $\left\{ e^{iX_{t}}\right\} _{t\in\mathbb{R}}$
where $\left\{ X_{t}\right\} _{t\in\mathbb{R}}$ is the standard Gaussian
process.

Let $X_{t}$ be the standard Brownian motion indexed by $t\in\mathbb{R}$;
i.e., $X_{0}=0$, $X_{t}~N(0,t)$, where$X_{t}$ is realized on a
probability space $\left(\Omega,\mathscr{A},\mathbb{P}\right)$, such
that, for all $s,t\in\mathbb{R}$, 
\begin{equation}
\mathbb{E}\left(X_{s}X_{t}\right)=\begin{cases}
\left|s\right|\wedge\left|t\right| & \text{if \ensuremath{s} and \ensuremath{t} have the same sign}\\
0 & \text{if \ensuremath{st\leq0}}.
\end{cases}\label{eq:e2}
\end{equation}
Here $\mathbb{E}\left(\cdot\right)$ denotes the expectation, 
\begin{equation}
\mathbb{E}\left(\cdots\right)=\int_{\Omega}\left(\cdots\right)d\mathbb{P}.
\end{equation}

\begin{rem}
The process $\left\{ X_{t}\right\} _{t\in\mathbb{R}}$ can be realized
in many different but equivalent ways. 
\end{rem}

Note that 
\begin{equation}
\mathbb{E}\left(\left|X_{s}-X_{t}\right|^{2}\right)=\left|s-t\right|,
\end{equation}
so the process $X_{t}$ has stationary and independent increments.
In particular, if $s,t>0$, then $\left|s-t\right|=s+t-2\,s\wedge t$,
and 
\begin{equation}
s\wedge t=\frac{s+t-\left|s-t\right|}{2}.
\end{equation}

\begin{prop}
The kernel $K$ from (\ref{eq:e1}) is positive definite. 
\end{prop}

\begin{proof}
Let $\left\{ X_{t}\right\} _{t\in\mathbb{R}}$ be the standard Brownian
motion and let $e^{iX_{t}}$ be the corresponding complex process,
then by direct calculation, 
\begin{equation}
e^{-\frac{1}{2}\left|t\right|}=\mathbb{E}\left(e^{iX_{t}}\right),\;\forall t\in\mathbb{R};\label{eq:e6}
\end{equation}
and 
\begin{eqnarray}
e^{-\frac{1}{2}\left|s-t\right|} & = & \mathbb{E}\left(e^{iX_{s}}e^{-iX_{t}}\right)\nonumber \\
 & = & \left\langle e^{iX_{s}},e^{iX_{t}}\right\rangle _{L^{2}\left(\mathbb{P}\right)},\;\forall s,t\in\mathbb{R}.\label{eq:e7}
\end{eqnarray}

The derivation of (\ref{eq:e6})-(\ref{eq:e7}) is based on power
series expansion of $e^{iX_{t}}$, and the fact that 
\begin{align}
\mathbb{E}\left(X_{t}^{2n}\right) & =\left(2n-1\right)!!\left|t\right|^{n},\;\text{and}\label{eq:e8}\\
\mathbb{E}\left(\left|X_{t}-X_{s}\right|^{2n}\right) & =\left(2n-1\right)!!\left|t-s\right|^{n},\label{eq:e9}
\end{align}
where 
\begin{equation}
\left(2n-1\right)!!=\left(2n-1\right)\left(2n-3\right)\cdots5\cdot3=\frac{\left(2n\right)!}{2^{n}n!}.\label{eq:e10}
\end{equation}

Now the positive definite property of $K$ follows from (\ref{eq:e7}),
since the RHS of (\ref{eq:e7}) is p.d. In details: for all $\left(c_{j}\right)_{j=1}^{N}$,
$c_{j}\in\mathbb{R}$: 
\begin{align*}
\sum\nolimits _{j}\sum\nolimits _{k}c_{j}c_{k}\,e^{-\frac{1}{2}\left|t_{j}-t_{k}\right|} & =\sum\nolimits _{j}\sum\nolimits _{k}c_{j}c_{k}\mathbb{E}\left(e^{iX_{t_{j}}}e^{-iX_{t_{k}}}\right)\\
 & =\mathbb{E}\left(\left|\sum\nolimits _{j}c_{j}e^{iX_{t_{j}}}\right|^{2}\right)\geq0.
\end{align*}
\end{proof}

\subsection{An Infinite-dimensional Fourier transform}
\begin{defn}
\label{def:gf} Let $\mathcal{H}\left(K\right)$ be the RKHS from
the kernel $K$ in (\ref{eq:e1}); and define the following transform
$\mathcal{T}:L^{2}\left(\mathbb{P}\right)\longrightarrow\mathcal{H}\left(K\right)$,
\begin{align}
\mathcal{T}\left(F\right)\left(t\right) & :=\mathbb{E}\left(e^{-iX_{t}}F\right)\nonumber \\
 & =\int_{\Omega}e^{-iX_{t}\left(\omega\right)}F\left(\omega\right)d\mathbb{P}\left(\omega\right)\label{eq:e11}
\end{align}
for all $F\in L^{2}\left(\mathbb{P}\right)$. Here, $\left\{ X_{t}\right\} $
is the Brownian motion. 

Our new Fourier transforms have properties in common with its classical
counterparts. It is isometric between the respective Hilbert spaces
but the Hilbert spaces are different in the infinite dimensional framework. 
\end{defn}

It is known that the standard Brownian motion, indexed by $\mathbb{R}$,
has a continuous realization (see e.g., \cite{MR562914}, page 46,
Theorem 2.1.) Hence the transform $\mathcal{T}$ defined by (\ref{eq:e11})
maps $L^{2}\left(\Omega,\mathbb{P}\right)$ into the bounded continuous
functions on $\mathbb{R}$. \corref{hl}, below, is the stronger assertion
that $\mathcal{T}$ maps $L^{2}\left(\Omega,\mathbb{P}\right)$ isometrically
onto the RKHS $\mathcal{H}\left(K\right)$ where $K$ is the kernel
in (\ref{eq:e1}).

Set 
\begin{equation}
K_{t}\left(\cdot\right)=e^{-\frac{1}{2}\left|t-\cdot\right|}\in\mathcal{H}\left(K\right),
\end{equation}
then by the reproducing property in $\mathcal{H}\left(K\right)$,
we have 
\begin{equation}
\left\langle K_{t},\psi\right\rangle _{\mathcal{H}\left(K\right)}=\psi\left(t\right),\;\forall\psi\in\mathcal{H}\left(K\right).
\end{equation}

\begin{lem}
Let $\mathcal{T}$ be the generalized Fourier transform in (\ref{eq:e11}),
and $\mathcal{T}^{*}$ be the adjoint operator; see the diagram below.
\begin{equation}
\xymatrix{L^{2}\left(\mathbb{P}\right)\ar@/^{1pc}/[r]^{\mathcal{T}} & \mathcal{H}\left(K\right)\ar@/^{1pc}/[l]^{\mathcal{T}^{*}}}
\label{eq:e14}
\end{equation}
Then, we have 
\begin{align}
\mathcal{T}\left(e^{iX_{t}}\right) & =K_{t},\;\text{and}\label{eq:e15}\\
\mathcal{T}^{*}\left(K_{t}\right) & =e^{iX_{t}}.\label{eq:e16}
\end{align}
\end{lem}

\begin{proof}
Recall the definition $\mathcal{T}\left(F\right)\left(s\right):=\mathbb{E}\left(e^{-iX_{s}}F\right)$.

\emph{Proof of }(\ref{eq:e15}). Setting $F=e^{iX_{t}}$, then 
\begin{align*}
\mathcal{T}\left(e^{iX_{t}}\right)\left(s\right) & =\mathbb{E}\left(e^{-iX_{s}}e^{iX_{t}}\right)=\mathbb{E}\left(e^{i\left(X_{t}-X_{s}\right)}\right)\\
 & =e^{-\frac{1}{2}\left|t-s\right|}=K_{t}\left(s\right).
\end{align*}

\emph{Proof of }(\ref{eq:e16})\emph{.} Let $F\in L^{2}\left(\mathbb{P}\right)$,
then 
\begin{align*}
\left\langle \mathcal{T}^{*}\left(K_{t}\right),F\right\rangle _{L^{2}\left(\mathbb{P}\right)} & =\left\langle K_{t},\mathcal{T}\left(F\right)\right\rangle _{\mathcal{H}\left(K\right)}=\mathcal{T}\left(F\right)\left(t\right)\\
 & =\mathbb{E}\left(e^{-iX_{t}}F\right)=\left\langle e^{iX_{t}},F\right\rangle _{L^{2}\left(\mathbb{P}\right)}.
\end{align*}
\end{proof}
\begin{cor}
\label{cor:hl}The generalized Fourier transform in (\ref{eq:e14})
is an isometric isomorphism from $L^{2}\left(\mathbb{P}\right)$ onto
$\mathcal{H}\left(K\right)$, with $\mathcal{T}\left(e^{iX_{t}}\right)=K_{t}$,
see (\ref{eq:e15}). Here, $\mathcal{H}\left(K\right)$ is a Hilbert
space that allows generalized spectral decompositions. 
\end{cor}

\begin{proof}
This is a direct application of (\ref{eq:e15})-(\ref{eq:e16}). Also
note that $\text{span}\left\{ e^{iX_{t}}\right\} $ is dense in $L^{2}\left(\mathbb{P}\right)$,
and $\text{span}\left\{ K_{t}\right\} $ is dense in $\mathcal{H}\left(K\right)$.
For the density of $\text{span}\left\{ e^{iX_{t}}\right\} $, see
e.g., \cite{MR1474726}, Lemma 2.7.
\end{proof}
\begin{conclusion}
$\mathcal{H}\left(K\right)$ is naturally isometrically isomorphic
to $L^{2}\left(\mathbb{P}\right)$. 
\end{conclusion}

\begin{rem}
To understand this isometric isomorphism $L^{2}\left(\mathbb{P}\right)\xrightarrow{\;\simeq\;}\mathcal{H}\left(K\right)$,
we must treat $L^{2}\left(\mathbb{P}\right)$ as a \emph{complex}
Hilbert space, while $\mathcal{H}\left(K\right)$ is defined as a
real Hilbert space; i.e., the generating functions $e^{iX_{t}}\in L^{2}\left(\mathbb{P}\right)$
are complex, where the inner product in $L^{2}\left(\mathbb{P}\right)$
is $\left\langle u,v\right\rangle _{L^{2}\left(\mathbb{P}\right)}=\int_{\Omega}\overline{u}vd\mathbb{P}$;
but the functions $K_{t}$, $t\in\mathbb{R}$, in $\mathcal{H}\left(K\right)$
are real valued. 
\end{rem}

\begin{rem}
Assume the normalization $X_{0}=0$, and $0<s<t$. The two processes
$X_{t-s}$, and $X_{t}-X_{s}$ are different, but they have the same
distribution $N\left(0,t-s\right)$. Indeed, we have 
\begin{eqnarray*}
\mathbb{E}\left(e^{iX_{t}}e^{-iX_{s}}\right) & = & \mathbb{E}\left(e^{i\left(X_{t}-X_{s}\right)}\right)\\
 & = & \sum_{n=0}^{\infty}\frac{i^{n}}{n!}\mathbb{E}\left(\left(X_{t}-X_{s}\right)^{n}\right)\\
 & = & \sum_{n=0}^{\infty}\frac{\left(-1\right)^{n}}{\left(2n\right)!}\mathbb{E}\left(\left(X_{t}-X_{s}\right)^{2n}\right)\quad(\text{since the odd terms cancel})\\
 & \underset{\text{by }\left(\ref{eq:e8}\right)}{=} & \sum_{n=0}^{\infty}\frac{\left(-1\right)^{n}}{\left(2n\right)!}\left(2n-1\right)!!\left(t-s\right)^{n}\\
 & \underset{\text{by }\left(\ref{eq:e9}\right)}{=} & \sum_{n=0}^{\infty}\frac{\left(-1\right)^{n}}{2^{n}n!}\left(t-s\right)^{n}\\
 & \underset{\text{by }\left(\ref{eq:e10}\right)}{=} & \sum_{n=0}^{\infty}\frac{1}{n!}\left(-\frac{1}{2}\left(t-s\right)\right)^{n}\\
 & = & e^{-\frac{1}{2}\left(t-s\right)}=\mathbb{E}\left(e^{iX_{t-s}}\right).
\end{eqnarray*}
\end{rem}

\begin{thm}
\label{thm:For-the-Gaussian}For the Gaussian kernel $K_{Gauss}\left(x,y\right)=e^{\frac{1}{2t}\left(x-y\right)^{2}}$,
we have 
\begin{align}
e^{-x^{2}/2t} & =\mathbb{E}\left(e^{ixX_{1/t}}\right),\;\text{and}\label{eq:e17}\\
e^{-\left(x-y\right)^{2}/2t} & =\mathbb{E}\left(e^{ixX_{1/t}}e^{-iyX_{1/t}}\right)=\left\langle e^{ixX_{1/t}},e^{iyX_{1/t}}\right\rangle _{L^{2}\left(\mathbb{P}\right)}.\label{eq:e18}
\end{align}
\end{thm}

\begin{proof}
A direct calculation yields 
\begin{align*}
\mathbb{E}\left(e^{ixX_{1/t}}\right) & =\sum_{n=0}^{\infty}\frac{\left(ix\right)^{n}}{n!}\mathbb{E}\left(X_{1/t}^{n}\right)\\
 & =\sum_{n=0}^{\infty}\frac{\left(-1\right)^{n}x^{2n}}{\left(2n\right)!}\left(2n-1\right)!!\frac{1}{t^{n}}\\
 & =\sum_{n=0}^{\infty}\frac{\left(-1\right)^{n}}{2^{n}n!}\frac{x^{2n}}{t^{n}}=\sum_{n=0}^{\infty}\frac{1}{n!}\left(-\frac{x^{2}}{2t}\right)^{n}=e^{-x^{2}/2t},
\end{align*}
which is (\ref{eq:e17}); and (\ref{eq:e18}) follows from this. 
\end{proof}
\begin{lem}
\label{lem:bs}Assume $0<s<t$. If $F\in L^{2}(\Omega,B_{s},P)$,
where $\mathscr{B}_{s}=$ $\sigma$-algebra generated by $\left\{ X_{u}\mathrel{;}u\leq s\right\} $,
then 
\begin{equation}
\mathcal{T}\left(F\right)\left(t\right)=e^{-\frac{t-s}{2}}\mathcal{T}\left(F\right)\left(s\right).\label{eq:bs1}
\end{equation}
\end{lem}

\begin{proof}
~ 
\begin{align*}
\text{LHS}_{\left(\ref{eq:bs1}\right)} & =\mathbb{E}\left(e^{-iX_{t}}F\right)\\
 & =\mathbb{E}\left(e^{-i\left(X_{t}-X_{s}\right)}e^{-iX_{s}}F\right)\\
\text{the independence of increament} & =\mathbb{E}\left(e^{-i\left(X_{t}-X_{s}\right)}\right)\mathbb{E}\left(e^{-iX_{s}}F\right)\\
 & =e^{-\frac{t-s}{2}}\mathcal{T}\left(F\right)\left(s\right)=\text{RHS}_{\left(\ref{eq:bs1}\right)}.
\end{align*}
\end{proof}
\begin{prop}
\label{prop:TXs}Let $0<s<t$, and let $H_{n}\left(\cdot\right)$,
$n\in\mathbb{N}_{0}$, be the Hermite polynomials; then 
\begin{equation}
\mathcal{T}\left(X_{s}^{n}\right)\left(t\right)=i^{n}e^{-\frac{t}{2}}s^{\frac{n}{2}}H_{n}\left(\sqrt{s}\right).\label{eq:bs2}
\end{equation}
\end{prop}

\begin{proof}
By \lemref{bs}, we have 
\begin{align*}
\mathcal{T}\left(X_{s}^{n}\right)\left(t\right) & =e^{-\frac{t-s}{2}}\mathcal{T}\left(X_{s}^{n}\right)\left(s\right)\\
 & =e^{-\frac{t-s}{2}}\mathbb{E}\left(e^{-iX_{s}}X_{s}^{n}\right)\\
 & =e^{-\frac{t-s}{2}}i^{n}\left(\frac{d}{d\lambda}\right)^{n}\big|_{\lambda=1}\mathbb{E}\left(e^{-i\lambda X_{s}}\right)\\
 & =e^{-\frac{t-s}{2}}i^{n}\left(\frac{d}{d\lambda}\right)^{n}\big|_{\lambda=1}\left(e^{-\frac{\lambda^{2}s}{2}}\right)\\
 & =i^{n}e^{-\frac{t-s}{2}}e^{-\frac{s}{2}}s^{\frac{n}{2}}H_{n}\left(\sqrt{s}\right)\\
 & =i^{n}e^{-\frac{t}{2}}s^{\frac{n}{2}}H_{n}\left(\sqrt{s}\right)
\end{align*}
which is the RHS in (\ref{eq:bs2}). In the calculation above, we
have used the following version of the Hermite polynomials $H_{n}\left(\cdot\right)$,
the probablist's variant; defined by

\[
\left(\frac{d}{d\xi}\right)^{n}e^{-\frac{\xi^{2}}{2}}=H_{n}\left(\xi\right)e^{-\frac{\xi^{2}}{2}}
\]
with the substitution $\xi=\sqrt{s}\lambda$ for $s>0$ fixed, and
$\lambda\rightarrow1$ $\Leftrightarrow$ $\xi\rightarrow\sqrt{s}$.

Note, the third ``$=$'' follows from standard calculation for the
generating functions of Gaussian random variables (see e.g., \cite{MR562914}):
\begin{align*}
\left(\frac{d}{d\lambda}\right)^{n}\mathbb{E}\left(e^{-i\lambda X_{s}}\right) & =\mathbb{E}\left(\left(-iX_{s}\right)^{n}e^{-i\lambda X_{s}}\right)\\
 & \Updownarrow\\
i^{n}\left(\frac{d}{d\lambda}\right)^{n}\mathbb{E}\left(e^{-i\lambda X_{s}}\right) & =\mathbb{E}\left(X_{s}^{n}e^{-i\lambda X_{s}}\right)
\end{align*}
and so 
\[
i^{n}\left(\frac{d}{d\lambda}\right)^{n}\big|_{\lambda=1}\mathbb{E}\left(e^{-i\lambda X_{s}}\right)=\mathbb{E}\left(X_{s}^{n}e^{-iX_{s}}\right).
\]
\end{proof}
\begin{acknowledgement*}
The co-authors thank the following colleagues for helpful and enlightening
discussions: Professors Daniel Alpay, Sergii Bezuglyi, Ilwoo Cho,
Wayne Polyzou, David Stewart, Eric S. Weber, and members in the Math
Physics seminar at The University of Iowa. 
\end{acknowledgement*}

\section*{Conflict of Interest Statement }

The authors did not receive support from any organization for the
submitted work. The authors have no relevant financial or non-financial
interests to disclose. 

Palle Jorgensen 

Myung-Sin Song 

James Tian

\bibliographystyle{amsalpha}
\bibliography{ref}

\end{document}